\documentclass[reqno,12pt]{amsart}
\def\bege{\begin{equation}} \def\ende{\end{equation}}
\def\begr{\begin{eqnarray}} \def\endr{\end{eqnarray}}
\usepackage{amsmath}
\usepackage{amssymb}
\usepackage{cite}

\usepackage{bbm}
\usepackage{amsmath}
\usepackage{txfonts}
\usepackage{stmaryrd}
\usepackage{amssymb}
\usepackage{amsfonts}
\usepackage{times}
\usepackage{mathrsfs}

\usepackage{amsthm}
\usepackage{enumerate}

\usepackage{framed}
\usepackage{lipsum}
\usepackage{color}

\usepackage[T1]{fontenc}    
\usepackage[utf8]{inputenc} 
\usepackage{babel}          
\babelprovide[import]{hungarian} 

\newcommand{\DD}{{\mathbb D}}

\def\D{\mathbb{D}}
\def\N{\mathbb N}

\def\a{\alpha}
\def\b{\beta}

\def\t{\theta}

\def\begr{\begin{eqnarray}} \def\endr{\end{eqnarray}}

\def\a{\alpha}\def\b{\beta}
\def\msk{\medskip}
\def\ol{\overline}

\def\msk{\medskip}
\def\ol{\overline}
\newtheorem{Lemma}{Lemma}[section]
\newtheorem{Theorem}[Lemma]{Theorem}

\newtheorem{Remark}[Lemma]{Remark}

\newcounter{other}            

\begin{document}
	\title[]{Banach algebra structure in Hardy-Carleson type tent spaces and Ces\`aro-like operators }
	
\author{Rong Yang and Xiangling Zhu$\dagger$}
	
	\address{Rong Yang
	\\ Institute of Fundamental and Frontier Sciences, University of Electronic Science and Technology of China, 610054, Chengdu, Sichuan, P.R. China.}
	\email{yangrong071428@163.com  }

\address{Xiangling Zhu
	\\ University of Electronic Science and Technology of China Zhongshan Institute, 528402, Zhongshan, Guangdong, P.R.
	China.}
    \email{jyuzxl@163.com}


	\begin{abstract} 
	In this paper,    the Hadamard-Bergman convolution and Banach algebra structure by the Duhamel product on Hardy-Carleson type tent spaces was investigated.
	Moreover,   the  boundedness and compactness of the Ces\`aro-like operator $\mathcal{C}_\mu$ on Hardy-Carleson type tent spaces $AT_p^\infty(\a)$ are also studied.

	\thanks{$\dagger$ Corresponding author.}
	\end{abstract}
	
	\maketitle

\section{Introduction}

Let $\mathbb{D}$ denote the open unit disk in the complex plane $\mathbb{C}$.
 Define $H(\mathbb{D})$ as the set of all analytic functions on $\mathbb{D}$.
For \(0 < p<\infty\), let \(H^{p}\) denote the Hardy space of all analytic functions \(f\in H(\D)\) such that
\[
\|f\|^p_{H^{p}}=\sup_{0<r<1}\frac{1}{2\pi}\int_{0}^{2\pi}|f(re^{it})|^{p}dt<\infty.
\]
The Bergman space \(A^{p}\) consists of all analytic functions \(f\in H(\D)\) for which
\[
\|f\|^p_{A^{p}}=\int_{\mathbb{D}}|f(z)|^{p}dA(z)<\infty,
\]
where \(dA(z) = \frac{1}{\pi}dxdy\) is the normalized Lebesgue area measure on \(\mathbb{D}\).

Let $\zeta > \frac{1}{2}$ and $\eta \in \mathbb{T}$, the boundary of $\mathbb{D}$.
The non-tangential approach region $\Gamma_\zeta(\eta)$ is defined by
\[
\Gamma(\eta)=\Gamma_\zeta(\eta) = \left\{z \in \mathbb{D} : |z - \eta| < \zeta(1 - |z|^2)\right\}.
\]
For $0<p<\infty$, the tent space $T_p^{\infty}(\alpha)$ consists of all measurable functions $f$ on $\mathbb{D}$ with
$$
\|f\|_{T_p^{\infty}(\alpha)}^p=\operatorname{esssup} _{\eta \in \mathbb{T}}\left(\sup _{u \in \Gamma(\eta)} \frac{1}{1-|u|^2} \int_{S(u)}|f(z)|^p(1-|z|^2)^{\a+1} d A(z)\right)<\infty,
$$
where 
$$S(r e^{i \theta})=\left\{\lambda e^{i t}: |t-\theta| \leq \frac{1-r}{2},1-\lambda \leq 1-r\right\}$$ for $r e^{i \theta} \in \mathbb{D} \backslash\{0\}$ and $S(0)=\mathbb{D}$.
Denote $T_p^\infty(\alpha) \cap H(\mathbb{D})$ by $A T_p^\infty(\alpha)$, called the Hardy-Carleson type tent space or the analytic tent space.  Tent spaces were initially introduced by Coifman, Meyer and Stein in \cite{cms} to address problems in harmonic analysis. They offered a general framework for examining questions concerning significant spaces, such as Bergman spaces and Hardy spaces.


Let \(f(z)=\sum_{n = 0}^{\infty}c_{n}z^{n}\) and \(g(z)=\sum_{n = 0}^{\infty}d_{n}z^{n}\). The Hadamard product \(f*g\) of functions $f$ and $g$ is defined as
\[
(f*g)(z)=\sum_{n = 0}^{\infty}c_{n}d_{n}z^{n},\quad z\in\mathbb{D}.
\]
It is a well-established fact that for \(f\in H^{1}\) and \(g\in H^{q}(1\leq q <\infty)\), 
\[
\lVert f*g\rVert_{H^{q}}\leq\lVert f\rVert_{H^{1}}\lVert g\rVert_{H^{q}}.
\]
Nevertheless, when \(f\in A^{1}\) and \(g\in A^{q}(1\leq q <\infty)\), the inequality
\[
\lVert f*g\rVert_{A^{q}}\leq\lVert f\rVert_{A^{1}}\lVert g\rVert_{A^{q}}
\]
is not satisfied. 
Karapetyants and Samko \cite{kas} introduced a modified form of the Hadamard product:
\[
f\widetilde{*}g(z)=\sum_{n = 0}^{\infty}\frac{c_{n}d_{n}}{n + 1}z^{n},\quad z\in\mathbb{D},
\]
which, in essence, represents a convolution in the sense that
\[
\mathbb{K}_{g}f(z)=\int_{\mathbb{D}}g(w)f(\overline{w}z)dA(w)=\sum_{n = 0}^{\infty}\mu_{n}c_{n}z^{n},
\]
where \(\mu_{n}=\frac{d_{n}}{n + 1}\) and \(\mathbb{K}_{g}\) is denoted as the Hadamard–Bergman convolution operator with kernel \(g\).
Moreover, in the Bergman space, the inequality
\[
\|\mathbb{K}_{g}f\|_{A^{p}}\leq\|f\|_{A^{p}}\|g\|_{A^1}
\]
is valid. 
A natural question arises: does the above inequality hold in the Hardy-Carleson type tent space \(AT_p^\infty(\a)\)? In this paper, we provide an affirmative answer to this question, establishing that
\[
\|\mathbb{K}_g f\|_{AT_p^\infty(\a)}
\lesssim \|f\|_{AT_p^\infty(\a)}\|g\|_{A^1}.
\]

As defined by Wigley (see \cite{wn1}), for analytic functions \(f\) and \(g\) on \(\mathbb{D}\), the Duhamel product \(f \circledast g\) is given by
\[
(f \circledast g)(z)=\frac{d}{dz}\int_{0}^{z}f(z - s)g(s)ds=\int_{0}^{z}f'(z - s)g(s)ds + f(0)g(z).
\]
This product has multiple applications such as operational calculus and boundary value problems. 
Wigley studied algebraic structures of analytic functions and maximal ideals in holomorphic function spaces and Hardy spaces \(H^p\) (\(p \geq 1\)) (see \cite{wn1,wn2}). 
The algebraic structure from the Duhamel product has been explored in different spaces. E.g., \cite{kms} for the Wiener algebra, \cite{kt06} for the space \(C^{(n)}(\D)\), \cite{ggs} for the Bergman space \(A^p\). 
For more Duhamel product results, see \cite{zlp,zlz} and  the references therein.


For $f(z) = \sum_{n=0}^\infty a_n z^n \in H(\mathbb{D})$,
the Ces\`aro operator $\mathcal{C}$ is given by  
\[
\mathcal{C}(f)(z) = \sum_{n=0}^\infty \left( \frac{1}{n+1} \sum_{k=0}^n a_k \right) z^n, \quad z \in \mathbb{D}.
\] 
The integral
form of  $\mathcal{C}$ is
$$ \mathcal{C}(f)(z)=\frac 1z\int_0^zf(\zeta)\frac 1{1-\zeta}d\zeta=\int_0^1 \frac{f(tz)}{1 - tz} dt .
$$ 
 Many researchers have explored the Ces\`aro operator on some analytic function spaces. For more details, see \cite{miao,sis0,sis1,sis2}.  Recently, Galanopoulos, Girela and Merch\'an,  as cited in \cite{ggm},   introduced the Ces\`aro-like operator $\mathcal{C}_\mu$.
For a finite positive Borel measure $\mu$ on $[0, 1)$, the Ces\`aro-like operator $\mathcal{C}_\mu$ is defined on $H(\mathbb{D})$ as follows:
\[
\mathcal{C}_\mu(f)(z) = \sum_{n=0}^\infty \left( \mu_n \sum_{k=0}^n a_k \right) z^n = \int_0^1 \frac{f(tz)}{1 - tz} d\mu(t), \quad z \in \mathbb{D},
\]
where $\mu_n$ stands for the moment of order $n$ of $\mu$, that is, $\mu_n = \int_0^1 t^n d\mu(t)$. 
They studied  the action of the operators \(\mathcal{C}_\mu\) on distinct spaces of analytic functions in \(\mathbb{D}\), such as the Hardy spaces \(H^{p}\), the weighted Bergman spaces \(A_{\alpha}^{p}\), \(BMOA\) (bounded mean oscillation of analytic functions), and the Bloch space \(\mathcal{B}\). Subsequently,    Bao etc. \cite{bsw} investigated the range of  Ces\`aro-like operator acting on the space \(H^{\infty}\), which consists of bounded analytic functions on $\mathbb{D}$.  To achieve this, they described the characterizations of Carleson type measures on the interval $[0,1)$.  In particular, they answered   an open question that was originally posed in \cite{ggm}. 
The Ces\`aro-like operator $\mathcal{C}_\mu$ has attracted a great deal of interest among numerous scholars.   
  See \cite{bsw,syz,tang23,tang, xll} and the references therein for more details.

%

In this   paper,  we will investigate the Hadamard-Bergman convolution on Hardy-Carleson type tent spaces. Moreover, we   give a Banach algebra structure by the Duhamel product for the Hardy-Carleson type tent space \(AT_p^\infty(\a)\). Finally, we  characterize the boundedness and compactness of the Ces\`aro-like operators \(C_\mu\)  on the  space \(AT_p^\infty(\alpha)\).  Specifically, we  prove that   \(\mathcal{C}_\mu\) is bounded (compact) on \(AT_p^\infty(\alpha)\) if and only if \(\mu\) is a Carleson measure (vanishing Carleson measure).


In this paper, we denote \(A\lesssim B\) to indicate the existence of a positive constant \(C\) such that \(A\leq CB\). Moreover, \(A\asymp B\) means that both \(A\lesssim B\) and \(B\lesssim A\) are valid.

\section{The Hadamard-Bergman convolution operators}
In this section, we describe the Hadamard-Bergman convolution on  Hardy-Carleson type tent spaces. To this end, we need some notations and lemmas.
For any $a\in\D$, let (see \cite{z1})
\[
\varphi_a(z) = \frac{a - z}{1 - \overline{a}z}, \quad z\in\D.
\]
It is obvious that $\varphi_a(z)$ is a M\"obius mapping that interchanges the points $0$ and $a$.
Let \(\varphi\) be an analytic self-map of \(\mathbb{D}\). The composition operator \(C_{\varphi}\) with symbol \(\varphi\) is defined by (see \cite{cm})
\[
C_{\varphi} f = f \circ \varphi.
\]

\begin{Lemma}\cite[Lemma 2.6]{ylh}\label{alp}
	Let $0<p<\infty$ and $\a>-2$.
	A function $f\in AT_p^\infty(\a)$ if and only if for each (or some) $t>0$,
	$$
	\sup_{a\in\D}\int_{\D}\frac{(1-|a|^2)^t}{|1-\ol{a}z|^{t+1}}|f(z)|^p(1-|z|^2)^{\a+1}dA(z)<\infty.
	$$
\end{Lemma}

According to \cite[Theorem 3.1]{zxlx}, we can obtain the following lemma.
\begin{Lemma}\label{th3.1}
Let $0<p<\infty$ and $\a>-2$. If $t\ge \frac{1}{p}$ and 
	\begin{align*}
	\sup_{u,a\in\D}\frac{(1-|u|^2)^t}{(|1-|\varphi_a(u)|^2)^\frac{1}{p}}\left(  \int_{\D}\frac{(1-|z|^2)^\a(1-|\varphi_a(z)|^2)}{|1-\ol{u}\varphi(z)|^{\a+2+pt}}dA(z) \right)^\frac{1}{p}<\infty,
\end{align*}
then $C_\varphi$ is a bounded operator on $AT_p^\infty(\a)$.
\end{Lemma}

\begin{Lemma}\cite[Lemma 2.5]{of2}\label{le2.5}
	For \(s > - 1,r,t\geq0\) and \(r + t - s>2\), we have
\begin{flalign*}
\quad \quad  &\int_{\D}\frac{(1-\vert\xi\vert^{2})^{s}}{\vert1-\overline{\xi}z\vert^{r}\vert1-\overline{\xi}w\vert^{t}}\mathrm{d}A(\xi)&&
\end{flalign*}	
	\begin{equation*}
	\lesssim\left\{ 
	\begin{aligned}
		&\frac{1}{\vert1-\overline{z}w\vert^{r + t - s - 2}},& \text{if }&r - s,t - s < 2\\
		&\frac{1}{(1-\vert z\vert^{2})^{r - s - 2}\vert1-\overline{z}w\vert^{t}},& \text{if }&t - s < 2<r - s\\
		&\frac{1}{(1-\vert z\vert^{2})^{r - s - 2}\vert1-\overline{z}w\vert^{t}}+\frac{1}{(1-\vert w\vert^{2})^{t - s - 2}\vert1-\overline{z}w\vert^{r}},& \text{if }&r - s,t - s>2.
	\end{aligned}
	\right.
	\end{equation*}
\end{Lemma}

\begin{Lemma}\label{l3.1a}
	Let $0<p<\infty$, $\a>-1$, $t> \frac{1}{p}$ and $f\in AT_p^\infty(\a)$. 
	Then
	\[
	\sup_{a\in\D}\int_{\D}\frac{(1-|a|^2)^t}{|1-\ol{a}z|^{t+1}}|f(\ol{w}z)|^{p}(1-|z|^{2})^{\a+1}dA(z)\lesssim \|f\|_{AT_p^\infty(\a)}^p
	\]
	for all \( w \in \mathbb{D}\).
\end{Lemma}

\begin{proof}
	Let   $\varphi(z)=\ol{w}z$. It is easy to check that $\varphi:\D \to \D$ is analytic. Using Lemma \ref{th3.1}, if 
	\begin{align*}
		\sup_{u,a\in\D}\frac{(1-|u|^2)^t}{(|1-|\varphi_a(u)|^2)^\frac{1}{p}}\left(  \int_{\D}\frac{(1-|z|^2)^\a(1-|\varphi_a(z)|^2)}{|1-\ol{u}\varphi(z)|^{\a+2+pt}}dA(z) \right)^\frac{1}{p}<\infty,
	\end{align*}
	we obtain
	\begin{align*}
		&\sup_{a\in\D}\int_{\D}\frac{(1-|a|^2)^t}{|1-\ol{a}z|^{t+1}}|f(\varphi(z))|^{p}(1-|z|^{2})^{\a+1}dA(z)\\
		\lesssim&  \sup_{a\in\D}\int_{\D}\frac{(1-|a|^2)^t}{|1-\ol{a}z|^{t+1}}|f(z)|^{p}(1-|z|^{2})^{\a+1}dA(z).
	\end{align*}
	Therefore, by Lemma \ref{alp},
	\[
	\sup_{a\in\D}\int_{\D}\frac{(1-|a|^2)^t}{|1-\ol{a}z|^{t+1}}|f(\ol{w}z)|^{p}(1-|z|^{2})^{\a+1}dA(z)\lesssim \|f\|_{AT_p^\infty(\a)}^p.
	\]
	So, we only need to prove that
	\begin{align*}
		\sup_{u,a\in\D}\frac{(1-|u|^2)^t}{(|1-|\varphi_a(u)|^2)^\frac{1}{p}}\left(  \int_{\D}\frac{(1-|z|^2)^\a(1-|\varphi_a(z)|^2)}{|1-\ol{u}\varphi(z)|^{\a+2+pt}}dA(z) \right)^\frac{1}{p}<\infty.
	\end{align*}
	Using Lemma \ref{le2.5}, we get
	\begin{align*}
		&\sup_{u,a\in\D}\frac{(1-|u|^2)^t}{(|1-|\varphi_a(u)|^2)^\frac{1}{p}}\left(  \int_{\D}\frac{(1-|z|^2)^\a(1-|\varphi_a(z)|^2)}{|1-\ol{u}\varphi(z)|^{\a+2+pt}}dA(z) \right)^\frac{1}{p}\\
		=&\sup_{u,a\in\D}\frac{(1-|u|^2)^t|1-\ol{a}u|^{\frac{2}{p}}}{(1-|a|^2)^\frac{1}{p}(1-|u|^2)^{\frac{1}{p}}}\left( \int_{\D}  \frac{(1-|a|^2)(1-|z|^2)^{\a+1}}{|1-\ol{a}z|^2|1-\ol{uw}z|^{\a+2+pt}} dA(z)\right)^\frac{1}{p}\\
		=&\sup_{u,a\in\D}(1-|u|^2)^{t-\frac{1}{p}}|1-\ol{a}u|^{\frac{2}{p}}\left( \int_{\D}  \frac{(1-|z|^2)^{\a+1}}{|1-\ol{a}z|^2|1-\ol{uw}z|^{\a+2+pt}} dA(z)\right)^\frac{1}{p}\\
		\lesssim &\sup_{u,a\in\D}\frac{(1-|u|^2)^{t-\frac{1}{p}}|1-\ol{a}u|^{\frac{2}{p}}}{(1-|uw|^2)^{t-\frac{1}{p}}
|1-\ol{a}uw|^{\frac{2}{p}}}<\infty,
	\end{align*}
as desired.  Here we used the assumption that  $\a>-1$ and $t> \frac{1}{p}$. 	The proof is complete.
\end{proof}

We now state and demonstrate the  main result in this section.
\begin{Theorem}
	Let  $1\le p<\infty$, \(\a>-1\), \(f \in AT_p^\infty(\a)\) and \(g \in A^1\). Then 
	\[
	\|\mathbb{K}_g f\|_{AT_p^\infty(\a)}
	\lesssim \|f\|_{AT_p^\infty(\a)}\|g\|_{A^1}.
	\]
\end{Theorem}

\begin{proof}
	Using Minkowski's inequality and Lemma \ref{l3.1a}, for $t> \frac{1}{p}$, we obtain
	\begin{align*}
		\|\mathbb{K}_g f\|_{AT_p^\infty(\a)}=&\sup_{a\in\D}\left(\int_{\D}\frac{(1-|a|^2)^t}{|1-\ol{a}z|^{t+1}}|\mathbb{K}_g f(z)|^{p}(1-|z|^{2})^{\a+1}dA(z)\right)^\frac{1}{p}\\
		\le &\sup_{a\in\D}\left(\int_{\D}\frac{(1-|a|^2)^t}{|1-\ol{a}z|^{t+1}}\left(\int_{\D}|f(\ol{w}z)||g(w)|dA(w) \right)^{p}(1-|z|^{2})^{\a+1}dA(z)\right)^\frac{1}{p}\\
		\le &\sup_{a\in\D}\int_{\D}|g(w)|\left( \int_{\D}\frac{(1-|a|^2)^t}{|1-\ol{a}z|^{t+1}}|f(\ol{w}z)|^{p}(1-|z|^{2})^{\a+1}dA(z) \right)^\frac{1}{p}dA(w)\\
		\lesssim&
		\|f\|_{AT_p^\infty(\a)}\int_{\D}|g(w)|dA(w)\\
		 \lesssim&\|f\|_{AT_p^\infty(\a)}\|g\|_{A^1}.
	\end{align*}
	The proof is complete.
\end{proof}

\section{Duhamel product}
In this section, we give a Banach algebra structure by the Duhamel product for Hardy-Carleson type tent spaces.
Therefore, we need some simple formulas for Duhamel products $\circledast$.
\[
\begin{aligned}
	(f \circledast g)(z) 
	&=\int_{0}^{z}f'(z - s)g(s)ds + f(0)g(z)\\
	&=\int_{0}^{z}f(z - s)g'(s)ds + f(z)g(0)\\
	&=\int_{0}^{z}g'(z - s)f(s)ds + g(0)f(z)
	=(g\circledast f)(z).
\end{aligned}
\]
It is obvious that Duhamel product is a commutative product.

If the integral line segment $[0, z]$ is halved, then integration by parts leads to
\begin{equation}\label{'}
	\begin{split}
		(f \circledast g)(z) &= \int_{0}^{\frac{z}{2}} f(z-s)g'(s)ds + \int_{0}^{\frac{z}{2}}g(z-s)f'(s) ds \\
		&\quad + f(z)g(0) + g(z)f(0) - f\left(\frac{z}{2}\right)g\left(\frac{z}{2}\right).
	\end{split}
\end{equation}

We   need the following lemmas.

\begin{Lemma}\label{az}
\cite[Lemma 2.7]{ylh}
Let $0<p<\infty$, $\a>-2$, $n\in\N \cup \{0\}$ and $f\in AT_p^\infty(\a)$. Then 
$$
|f^{(n)}(z)|\lesssim\frac{\|f\|_{AT_p^\infty(\a)}}{(1-|z|^2)^{\frac{\a+2}{p}+n}}
$$
for all $z\in\D$.
	
\end{Lemma}

\begin{Remark}\label{re2a}
From Lemma \ref{az}, if $K$ is a compact subset of $\D$, there exists a constant $C>0$ such that
$$
|f(z)|\leq C\|f\|_{AT_p^\infty(\a)}~~ {\text and}~~
|f'(z)|\leq C\|f\|_{AT_p^\infty(\a)}
$$
for all $z\in K$.
\end{Remark}

\begin{Lemma}\label{le2a}
Let $0<p<\infty$, $\a>-2$, and $f\in AT_p^\infty(\a)$. Then
$$
\sup_{a\in\D}\int_{\D}\frac{(1-|a|^2)^t}{|1-\ol{a}z|^{t+1}}\left|\int_{0}^{\frac{z}{2}}|f(z-s)||ds|\right|^{p}(1-|z|^{2})^{\a+1}dA(z)\le C\|f\|_{AT_p^\infty(\a)}^p
$$
for any $t>0$.

\end{Lemma}

\begin{proof}
Let $z=re^{i\t}$ and $s=\rho e^{i\t}$, $0\le \rho \le \frac{r}{2}$. Then,
\[
\int_{0}^{\frac{z}{2}}|f(z - s)||ds|=\int_{\frac{r}{2}}^{r}|f(\rho e^{i\theta})|d\rho\leq\int_{0}^{r}|f(\rho e^{i\theta})|d\rho.
\]
Applying Lemma \ref{az}, we deduce that
\begin{align*}
	\int_{0}^{\frac{z}{2}}|f(z - s)||ds|
	\leq&\|f\|_{AT_p^\infty(\a)}\int_{0}^{r}\frac{d\rho}{(1-\rho)^{\frac{\a+2}{p}}}\\
\le&  \begin{cases} 
	\|f\|_{AT_p^\infty(\a)} \left( \frac{1 - (1 - r)^{1-\frac{\a+2}{p}}}{1-\frac{\a+2}{p}  } \right), & 1-\frac{\a+2}{p} \neq 0 \\
	\|f\|_{AT_p^\infty(\a)} |\ln(1 - r)| ,& 1-\frac{\a+2}{p} = 0.
\end{cases}
\end{align*}
Hence, when $1-\frac{\a+2}{p} \neq 0$, for any $t>0$, we obtain
\[
\begin{aligned}
&\sup_{a\in\D}\int_{\D}\frac{(1-|a|^2)^t}{|1-\ol{a}z|^{t+1}}\left|\int_{0}^{\frac{z}{2}}|f(z-s)||ds|\right|^{p}(1-|z|^{2})^{\a+1}dA(z)\\
\leq &C\|f\|_{AT_p^\infty(\a)}^p
\sup_{a\in\D}\int_{\D}\frac{(1-|a|^2)^t}{|1-\ol{a}z|^{t+1}}(1-|z|^2)^{p-1}dA(z)\\
\leq& C\|f\|_{AT_p^\infty(\a)}^p.
\end{aligned}
\]
When $1-\frac{\a+2}{p}=0$, for any $t>0$, we have
\[
\begin{aligned}
	&\sup_{a\in\D}\int_{\D}\frac{(1-|a|^2)^t}{|1-\ol{a}z|^{t+1}}\left|\int_{0}^{\frac{z}{2}}|f(z-s)||ds|\right|^{p}(1-|z|^{2})^{\a+1}dA(z)\\
	\le& C\|f\|_{AT_p^\infty(\a)}^p\sup_{a\in\D}\int_{\D}\frac{(1-|a|^2)^t}{|1-\ol{a}z|^{t+1}}|\ln(1-|z|^2)|^{p}dA(z)\\
	\le& C\|f\|_{AT_p^\infty(\a)}^p,
\end{aligned}
\]
as desired. The proof is complete.
\end{proof}

Now, we state and prove our main result in this section.

\begin{Theorem}\label{a1.1}
Let $1\le p<\infty$ and $\a>-2$. The Hardy-Carleson type tent space \(AT_p^\infty(\a)\) is a unital (the unit here is the constant function \(1\)) commutative Banach algebra with respect to Duhamel product \(\circledast\). Denote the algebras as \((AT_p^\infty(\a),\circledast)\).

\end{Theorem}
\begin{proof}
Let $f,g\in AT^\infty_p(\a)$.
Using $(\ref{'})$ and Remark \ref{re2a}, we get
\[
\begin{aligned}
	|(f \circledast g)(z)|\leq
	&C\|g\|_{AT_p^\infty(\a)}\int_{0}^{\frac{z}{2}}|f(z - s)||ds|+C\|f\|_{AT_p^\infty(\a)}\int_{0}^{\frac{z}{2}}|g(z-s)||ds|\\
	&+C\|g\|_{AT_p^\infty(\a)}|f(z)|+C\|f\|_{AT_p^\infty(\a)}|g(z)|+C\|f\|_{AT_p^\infty(\a)}\|g\|_{AT_p^\infty(\a)}.
\end{aligned}
\]
When \(1\le p<\infty\), using Lemma \ref{alp} and Minkowski’s inequality, for any $t>0$, we obtain
\[
\begin{aligned}
&\|(f \circledast g)\|_{AT_p^\infty(\a)}^p\\
=&\sup_{a\in\D}\int_{\D}\frac{(1-|a|^2)^t}{|1-\ol{a}z|^{t+1}}| (f \circledast g)(z) |^{p}(1-|z|^{2})^{\a+1}dA(z)\\
\leq&
C\|g\|_{AT_p^\infty(\a)}^p\sup_{a\in\D}\int_{\D}\frac{(1-|a|^2)^t}{|1-\ol{a}z|^{t+1}}\left|\int_{0}^{\frac{z}{2}}|f(z-s)||ds|\right|^{p}(1-|z|^{2})^{\a+1}dA(z)\\
&+C\|f\|_{AT_p^\infty(\a)}^p\sup_{a\in\D}\int_{\D}\frac{(1-|a|^2)^t}{|1-\ol{a}z|^{t+1}}\left|\int_{0}^{\frac{z}{2}}|g(z-s)||ds|\right|^{p}(1-|z|^{2})^{\a+1}dA(z)\\
&+C\|f\|_{AT_p^\infty(\a)}^p\|g\|_{AT_p^\infty(\a)}^p.
\end{aligned}
\]
Applying Lemma \ref{le2a}, we have
\[
\|(f \circledast g)\|_{AT_p^\infty(\a)}\leq C\|f\|_{AT_p^\infty(\a)} \|g\|_{AT_p^\infty(\a)}.
\]
The proof is complete.

\end{proof}

Next we establish a Young type property for the Duhamel convolution operator with analytic symbol \(f\):
\[
\mathfrak{D}_{f}g(z)=\int_{0}^{z}f'(z - s)g(s)+f(0)g(z), \quad g\in AT_p^\infty(\a).
\]

\begin{Theorem}
	Let $1\le p,q<\infty$, $\a,\b>-2$ and \(f\in AT_q^\infty(\b)\). Then
	\begin{enumerate}
		\item [(i)] If $1-\frac{\b+2}{q}<0$, then \(\mathfrak{D}_{f}\in\mathscr{B}(AT_p^\infty(\a))\) for any  $  p<\frac{\a+2}{\frac{\b+2}{q}-1}$.
		
		\item [(ii)] If $1-\frac{\b+2}{q}\ge0$, then \(\mathfrak{D}_{f}\in\mathscr{B}(AT_p^\infty(\a))\) for all $p\ge1$. 
		
	\end{enumerate}
\end{Theorem}

\begin{proof}
	Let \(f\in AT_q^\infty(\b) \) and \(g\in AT_p^\infty(\a) \). Using  $(\ref{'})$ and Remark \ref{re2a}, we get
	\[
	\begin{aligned}
		|(f \circledast g)(z)|\leq
		&C\|g\|_{AT_p^\infty(\a)}\int_{0}^{\frac{z}{2}}|f(z - s)||ds|+C\|f\|_{AT_p^\infty(\a)}\int_{0}^{\frac{z}{2}}|g(z-s)||ds|\\
		&+C\|g\|_{AT_p^\infty(\a)}|f(z)|+C\|f\|_{AT_p^\infty(\a)}|g(z)|+C\|f\|_{AT_p^\infty(\a)}\|g\|_{AT_p^\infty(\a)}.
	\end{aligned}
	\]
	Using Lemma \ref{az}, we obtain
	\begin{align*}
		 \int_{0}^{\frac{z}{2}}|f(z - s)||ds|
		\leq&\|f\|_{AT_q^\infty(\b)}\int_{0}^{r}\frac{d\rho}{(1-\rho)^{\frac{\b+2}{q}}}\\
		\le&  \begin{cases} 
			\|f\|_{AT_q^\infty(\b)} \left( \frac{1 - (1 - r)^{1-\frac{\b+2}{q}}}{1-\frac{\b+2}{q}  } \right), & 1-\frac{\b+2}{q} \neq 0, \\
			\|f\|_{AT_q^\infty(\b)} |\ln(1 - r)| ,& 1-\frac{\b+2}{q} = 0.
		\end{cases}
	\end{align*}
	Therefore, when $1-\frac{\b+2}{q} \neq 0$ and $p\left(1-\frac{\b+2}{q} \right)+\a+1>-1$, for any $t>0$, by Lemma \ref{alp} or Lemma 3.10 in \cite{z1} we have
	\[
	\begin{aligned}
		&\sup_{a\in\D}\int_{\D}\frac{(1-|a|^2)^t}{|1-\ol{a}z|^{t+1}}\left|\int_{0}^{\frac{z}{2}}|f(z-s)||ds|\right|^{p}(1-|z|^{2})^{\a+1}dA(z)\\
		\leq&C\|f\|_{AT_q^\infty(\b)}^p
		\sup_{a\in\D}\int_{\D}\frac{(1-|a|^2)^t}{|1-\ol{a}z|^{t+1}}(1-|z|^2)^{p\left(1-\frac{\b+2}{q} \right)+\a+1}dA(z)\\
		\leq& C\|f\|_{AT_q^\infty(\b)}^p. 
	\end{aligned}
	\]
	When $1-\frac{\b+2}{q}=0$, for any $t>0$, we get
	\[
	\begin{aligned}
		&\sup_{a\in\D}\int_{\D}\frac{(1-|a|^2)^t}{|1-\ol{a}z|^{t+1}}\left|\int_{0}^{\frac{z}{2}}|f(z-s)||ds|\right|^{p}(1-|z|^{2})^{\a+1}dA(z)\\
		\le& C\|f\|_{AT_q^\infty(\b)}^p\sup_{a\in\D}\int_{\D}\frac{(1-|a|^2)^t}{|1-\ol{a}z|^{t+1}}|\ln(1-|z|^2)|^{p}dA(z)\\
		\le& C\|f\|_{AT_q^\infty(\b)}^p.
	\end{aligned}
	\]
	Repeating the steps of the proof of Theorem \ref{a1.1}, for any $t>0$, we deduce that
	\[
	\begin{aligned}
		\|\mathfrak{D}_{f}g\|_{AT_p^\infty(\a)}^p
		\leq&C\|g\|_{AT_p^\infty(\a)}^p\sup_{a\in\D}\int_{\D}\frac{(1-|a|^2)^t}{|1-\ol{a}z|^{t+1}}\left|\int_{0}^{\frac{z}{2}}|f(z-s)||ds|\right|^{p}(1-|z|^{2})^{\a+1}dA(z)\\
		&+C\|f\|_{AT_q^\infty(\b)}^p\sup_{a\in\D}\int_{\D}\frac{(1-|a|^2)^t}{|1-\ol{a}z|^{t+1}}\left|\int_{0}^{\frac{z}{2}}|g(z-s)||ds|\right|^{p}(1-|z|^{2})^{\a+1}dA(z)\\
		&+C\|f\|_{AT_q^\infty(\b)}^p\|g\|_{AT_p^\infty(\a)}^p.
	\end{aligned}
	\]
	Applying Lemma \ref{le2a}, it follows  that 
	\[
	\|\mathfrak{D}_{f}g\|_{AT_p^\infty(\a)}\leq C\|f\|_{AT_q^\infty(\b)}\|g\|_{AT_p^\infty(\a)}.
	\]
\end{proof}

\section{ Boundedness and compactness of $C_\mu:AT_p^\infty(\a)\to AT_p^\infty(\a)$}

To prove the main result in this section, we  need some notations and lemmas. 
%
%

Let \(\mu\) denote a positive Borel measure defined on \(\mathbb{D}\) and   \(s > 0\). The measure \(\mu\) is referred to as an \(s\)-Carleson measure on \(\mathbb{D}\) provided that (see \cite{ca2})  
\[
\|\mu\|_{s}=\sup_{I\subset \partial\mathbb{D}}\frac{\mu(S(I))}{|I|^{s}}<\infty.
\]
Here,
$
S(I)=\left\{ z=re^{i\theta}\in \DD:1-|I|\leq r<1, e^{i\theta}\in I \right\}. 
$
In particular, the $1$-Carleson measure coincides with the classical Carleson measure.

A positive Borel measure $\mu$ on  $[0,1)$ can be regarded as a Borel measure on  $\mathbb{D}$ by establishing an identification with the measure $\overline{\mu}$. The measure $\overline{\mu}$ is defined as follows: for every Borel subset $E$ of $\mathbb{D}$,
\[
\overline{\mu}(E)=\mu(E\cap [0,1))
\] 
Consequently, the measure \(\mu\) is an \(s\)-Carleson measure  on \([0, 1)\) if there exists a constant \(C > 0\) such that (see \cite{bsw})  
\[
\mu([t, 1)) \leq C (1 - t)^s, \quad 0 \leq t < 1.
\]
The measure \(\mu\) is a vanishing \(s\)-Carleson measure  on \([0, 1)\) if
\[
\lim_{t\rightarrow 1}\frac{\mu([t, 1))}{(1 - t)^s} =0.
\]

The following characterization of Carleson measures on $[0, 1)$ is due to Bao et al. (see \cite[Proposition 2.1]{bsw}).

\begin{Lemma}\label{le2.4}
Suppose \(r>0\), \(0 \leq c<s<\infty\) and \(\mu\) is a finite positive Borel measure on $[0, 1)$. Then the following statements are equivalent:

\begin{enumerate}
\item [(i)] \(\mu\) is a \(s\)-Carleson measure;
	
\item [(ii)]
\[\sup _{b \in \mathbb{D}} \int_{0}^{1} \frac{(1-|b|)^{r}}{(1-t)^{c}(1-|b| t)^{s+r-c}} d \mu(t)<\infty;\]
	
\item [(iii)]
\[\sup _{b \in \mathbb{D}} \int_{0}^{1} \frac{(1-|b|)^{r}}{(1-t)^{c}|1-b t|^{s+r-c}} d \mu(t)<\infty.\]
\end{enumerate}
\end{Lemma}

For vanishing Carleson measures on $[0, 1)$, we have the following result (see \cite[Lemma 4.2]{tang} or \cite[Lemma 2.3]{xll}). 

\begin{Lemma}\label{le4.2}
	Suppose \(r>0\), \(0 \leq c<s<\infty\) and \(\mu\) is a finite positive Borel measure on $[0, 1)$. Then the following statements are equivalent:
	\begin{enumerate}
		\item [(i)] \(\mu\) is a vanishing \(s\)-Carleson measure;
		
		\item [(ii)] 
		\[\lim _{ |b| \to 1^- } \int_{0}^{1} \frac{(1-|b|)^{r}}{(1-t)^{c}(1-|b| t)^{s+r-c}} d \mu(t)=0;\]
		
		\item [(iii)]
		\[\lim _{ |b| \to 1^- } \int_{0}^{1} \frac{(1-|b|)^{r}}{(1-t)^{c}|1-b t|^{s+r-c}} d \mu(t)=0.\]
	\end{enumerate}
\end{Lemma}

The next two lemmas are very useful  in the proof of our main results in this paper.

\begin{Lemma}\cite[Proposition 3.1]{zlt}\label{le2.2}
Let \(w, a \in \mathbb{D}\). For \(r>0\) and \(t>0\), let 
\[I_{w, a}=\int_{0}^{2 \pi} \frac{1}{\left|1-\overline{w} e^{i \theta}\right|^{t}\left|1-\overline{a} e^{i \theta}\right|^{r}} d \theta. \]
Then the following results hold:

\begin{enumerate}
%
%
%
%
%

\item [(i)] When \(t > 1\) and \(r > 1\), 
\[I_{w, a} \asymp \frac{1}{(1-|w|^{2})^{t-1}|1-w \bar{a}|^{r}}+\frac{1}{(1-|a|^{2})^{r-1}|1-w \bar{a}|^{t}}.\]
	
\item [(ii)] When \(t > 1 = r\), 
\[ I_{w,a} \asymp \frac{1}{(1 - |w|^2)^{t - 1} |1 - \overline{w}a|} + \frac{1}{|1 - \overline{w}a|^t} \log \frac{e}{1 - |\varphi_w(a)|^2}. \] 
\end{enumerate}
\end{Lemma}

\begin{Lemma}\cite[Lemma 2.2]{zgsl}\label{le2.3}
Let \(\delta>-1\), \(c>0\), \(0 \leq \rho<1\). Then 
\[\int_{0}^{1} \frac{(1-r)^{\delta}}{(1-\rho r)^{\delta+c+1}} d r \asymp \frac{1}{(1-\rho)^{c}}.\]
\end{Lemma}

The following lemma gives an equivalent characterization  for functions in $AT_p^\infty(\a)$ and can be found in \cite{cw}.
\begin{Lemma}\label{at}
Let $0<p<\infty$ and $\a>-2$.
Then $g\in AT_p^\infty(\a)$ if and only if 
$$
\sup_{b \in \mathbb{D}}  \int_{\mathbb{D}} |g'(w)|^p (1 - |w|^2)^{p+\a} (1 - |\varphi_b(w)|^2) dA(w)<\infty.
$$
\end{Lemma}

The following lemma is useful for studying compactness. Its proof is similar to Proposition 3.11 in \cite{cm}, so details are omitted. 

\begin{Lemma}\label{com}	Let $0<p<\infty$ and $\a>-2$.	Let \(T : AT_p^\infty(\a) \to AT_p^\infty(\a)\) be a bounded linear operator. Then \(T \) is compact if and only if for any bounded sequence \(\{f_j\}\) in \(AT_p^\infty(\a)\) which converges to zero uniformly on compact subsets of \(\mathbb{D}\), 
	\[\lim _{j \to \infty}\left\|Tf_{j}\right\|_{AT_p^\infty(\a)}=0.\]
\end{Lemma}

Now, we are in a position to state and prove the main results of this section. Specifically, we will characterize the boundedness and compactness of the Ces\`aro-like operator \(\mathcal{C}_\mu\) which maps from \(AT_p^\infty(\a)\) to \(AT_p^\infty(\a)\).

\begin{Theorem}\label{cmub}
Let \(1 \leq p<\infty\) and \(\a>-2\). Let \(\mu\) be a finite positive Borel measure on $[0, 1)$. The Ces\`aro-like operator \(\mathcal{C}_\mu\) is bounded on \(AT_p^\infty(\a)\) if and only if \(\mu\) is a Carleson measure.
\end{Theorem}

\begin{proof} First we assume that \(\mathcal{C}_\mu\) is bounded on \(AT_p^\infty(\a)\). For $0<\rho<1$,  let 
\[f_{\rho}(z)=\frac{(1-\rho)}{(1-\rho z)^{ \frac{\a+2}{p}+1}}, \quad z\in \D.\]
After calculation, we see that \(f_{\rho} \in AT_p^\infty(\a)\) and \(\sup _{0<\rho<1}\left\|f_{\rho}\right\|_{AT_p^\infty(\a)} \lesssim 1\). This implies that \(\mathcal{C}_\mu(f_{\rho}) \in AT_p^\infty(\a)\).
Using Lemma \ref{az}, we have 
\[
\left|\mathcal{C}_{\mu}\left(f_{\rho}\right)(\rho)\right| \lesssim \frac{1}{(1-\rho)^{ \frac{\a+2}{p}}},\quad 0<\rho<1.
\]
Then, for \(\frac{1}{2}<\rho<1\), 
\[\begin{aligned} 
\frac{1}{(1-\rho)^{ \frac{\a+2}{p}}} & \gtrsim\left|\int_{0}^{1} \frac{(1-\rho)}{(1-t \rho)\left(1-t \rho^{2}\right)^{ \frac{\a+2}{p}+1}} d \mu(t)\right| \\ 
& \gtrsim \int_{\rho}^{1} \frac{(1-\rho)}{(1-t \rho)\left(1-t \rho^{2}\right)^{ \frac{\a+2}{p}+1}} d \mu(t)   \gtrsim \frac{\mu([\rho, 1))}{(1-\rho)^{ \frac{\a+2}{p}+1}},
\end{aligned}\]
which implies that $\mu([\rho, 1)) \lesssim 1-\rho $
for all $\frac{1}{2}<\rho<1$.  Hence, \(\mu\) is a Carleson measure.
	
Conversely, assume that \(\mu\) is a Carleson measure.
Let $f\in AT_p^\infty(\a)$.
Without loss of generality, we may assume \(f(0)=0\). By Lemma \ref{az},
we get 
\[\begin{aligned} 
 \left|\mathcal{C}_{\mu}(f)'(z)\right|  
=&\left|\int_{0}^{1} \frac{t f'(t z)}{1-t z}+\frac{t f(t z)}{(1-t z)^{2}} d \mu(t)\right| \\ 
\leq &\int_{0}^{1} \frac{\left|t f'(t z)\right|}{|1-t z|} d \mu(t)+\int_{0}^{1} \frac{|t f(t z)|}{|1-t z|^{2}} d \mu(t) \\ 
\lesssim& \| f\| _{AT_p^\infty(\a)}\left( \int_{0}^{1} \frac{d \mu(t)}{|1-t z|(1-t|z|)^{ \frac{\a+2}{p}+1}}+ \int_{0}^{1} \frac{d \mu(t)}{|1-t z|^{2}(1-t|z|)^{ \frac{\a+2}{p}}}\right) \\ 
\lesssim&\| f\| _{AT_p^\infty(\a)} \int_{0}^{1} \frac{d \mu(t)}{|1-t z|(1-t|z|)^{ \frac{\a+2}{p}+1}}. 
\end{aligned}\]
Using Lemma \ref{at} and Minkowski’s inequality, we obtain
\[
\begin{aligned}
 &\| \mathcal{C}_\mu (f) \|_{AT_p^\infty(\a)}^p  
\asymp  \sup_{b \in \mathbb{D}}  \int_{\mathbb{D}} |\mathcal{C}_\mu (f)'(z)|^p (1 - |z|^2)^{p+\a} (1 - |\varphi_b(z)|^2) dA(z)  \\
\lesssim& \| f \|_{AT_p^\infty(\a)}^p \sup_{b \in \mathbb{D}} (1 - |b|^2)  \int_{\mathbb{D}} \left( \int_0^1 \frac{d\mu(t)}{|1 - tz|(1 - t|z|)^{ \frac{\a+2}{p}+1}} \right)^p \frac{(1 - |z|^2)^{p+\a+1}}{|1 - \overline{b}z|^{2}} dA(z) \\
\lesssim& \| f \|_{AT_p^\infty(\a)}^p \sup_{b \in \mathbb{D}} (1-|b|^2) \left( \int_0^1 \left( \int_{\mathbb{D}} \frac{(1 - |z|^2)^{p+\a+1} dA(z)}{|1 - tz|^p(1 - t|z|)^{p+\a+2}  |1 - \overline{b}z|^{2}} \right)^{\frac{1}{p}} d\mu(t) \right)^{p}.
\end{aligned}
\] 	
To complete the proof, it suffices to prove 
\[
\sup_{b \in \mathbb{D}} (1-|b|^2) \left( \int_0^1 \left( \int_{\mathbb{D}} \frac{(1 - |z|^2)^{p+\a+1} dA(z)}{|1 - tz|^p(1 - t|z|)^{p+\a+2}  |1 - \overline{b}z|^{2}} \right)^{\frac{1}{p}} d\mu(t) \right)^{p}< \infty.
\]

{\bf Case \( p > 1 \).}
Using the polar coordinate formula and Lemma \ref{le2.2} (i), we get
\[
\begin{aligned}
&\int_{\mathbb{D}} \frac{(1 - |z|^2)^{p+\a+1} }{|1 - tz|^p(1 - t|z|)^{p+\a+2}  |1 - \overline{b}z|^{2}}dA(z) \\
\asymp& \int_0^1 \frac{(1 - r)^{p+\a+1}}{(1 - tr)^{p+\a+2}} \int_0^{2\pi} \frac{d\theta}{|1 - tre^{i\theta}|^p |1 - \overline{b}re^{i\theta}|^{2}} dr \\
\asymp& \int_0^1 \frac{(1 - r)^{p+\a+1}}{(1 - tr)^{p+\a+2}} \left( \frac{1}{(1 - tr)^{p - 1} |1 - t\overline{b}r^2|^{2}} + \frac{1}{(1 - |b|r) |1 - t\overline{b}r^2|^p} \right) dr \\
=&\int_{0}^{1}\frac{(1-r)^{p+\a+1}}{(1-tr)^{2p+\a+1}|1-t\ol{b}r^2|^2}dr+\int_0^1 \frac{(1 - r)^{p+\a+1}}{(1 - tr)^{p+\a+2}(1-|b|r)|1-t\ol{b}r^2|^p}dr\\
=& J_1+J_2.
\end{aligned}
\] 
Using the fact that 
\begin{align}\label{xy}
(x + y)^p \leqslant 
\begin{cases} 
	x^p + y^p, & 0<p<1, \\
	2^{p-1}(x^p + y^p), & p \ge 1 
\end{cases}, \quad x,y>0,
\end{align}
we obtain
\[
\begin{aligned}
&\sup_{b \in \mathbb{D}} (1-|b|^2) \left( \int_0^1 \left( \int_{\mathbb{D}} \frac{(1 - |z|^2)^{p+\a+1} dA(z)}{|1 - tz|^p(1 - t|z|)^{p+\a+2}  |1 - \overline{b}z|^{2}} \right)^{\frac{1}{p}} d\mu(t) \right)^{p}\\
\lesssim& \sup_{b \in \mathbb{D}} (1-|b|^2) \left(\int_0^1 (J_1 + J_2)^{\frac{1}{p}} d\mu(t)\right)^p \\
\lesssim& \sup_{b \in \mathbb{D}} (1-|b|^2) \left(\int_0^1 J_1^{\frac{1}{p}} d\mu(t)\right)^p + \sup_{b \in \mathbb{D}} (1-|b|^2) \left(\int_0^1 J_2^{\frac{1}{p}} d\mu(t)\right)^p.
\end{aligned}
\]
It is easy to see that \( 1 - r < 1 - tr \), \( 1 - t < 1 - tr \) and \(|1-t\ol{b}r^2|\ge 1-t|b|r\).
Therefore, we can choose a positive real number \( \epsilon \) that satisfies
 \[ \frac{p-1}{p}< \epsilon <1. \] 
By the choice of \( \epsilon \), it is easy to check that \( p\epsilon-p> -1 \) and \( 2 - (p\epsilon-p) - 1 = 1-p\epsilon+p > 0 \). 
Using Lemma \ref{le2.3}, we get
\[
\begin{aligned}
J_1 &= \int_0^1 \frac{(1 - r)^{p+\a+1}}{(1 - tr)^{2p+\a+1} |1 - t\overline{b}r^2|^{2}}dr \\
&\leq \int_0^1 \frac{1}{(1 - tr)^{p} |1 - t\overline{b}r^2|^{2}}dr \\
&\leq \int_0^1 \frac{1}{(1 - t)^{p\epsilon} (1 - tr)^{p-p\epsilon} |1 - t\overline{b}r^2|^{2}}dr \\
&\lesssim \frac{1}{(1 - t)^{p\epsilon}} \int_0^1 \frac{(1 - r)^{p\epsilon-p}}{(1 - t|b|r)^{2}} dr\\
&\lesssim \frac{1}{(1 - t)^{p\epsilon} (1 - t|b|)^{1-p\epsilon+p}}.
\end{aligned}
\]
Since $\mu$ is a Carleson measure, using Lemma \ref{le2.4}, we have
\[
\begin{aligned}
&\sup_{b \in \mathbb{D}} (1-|b|^2) \left(\int_0^1 J_1^{\frac{1}{p}} d\mu(t) \right)^p\\
\lesssim &
\sup_{b \in \mathbb{D}} \left( \int_0^1 \frac{(1 - |b|)^{\frac{1}{p}}}{(1 - t)^{\epsilon} (1 - t|b|)^{\frac{1}{p}-\epsilon+1}} d\mu(t) \right)^p
\lesssim 1.
\end{aligned}
\]   
Subsequently, we focus our efforts on the estimation of \(J_2\).

Let \( 0 < \delta <1\) 
and 
\( 0 < \tau < \frac{1}{p} \). We may choose \( \delta \) and \( \tau \) such that \( \frac{1}{p} < \delta + \tau < \frac{p+1}{p} \). Notice that \( 1 -r  < 1 - |b|r \) and \( 1 - |b| < 1 - |b|r \). By the choices of \( \delta \) and  \( \tau \), it is easy to see that \( p(\delta+\tau)-2> -1 \) and \( p-[p(\delta+\tau)-2]-1 = p-p(\delta+\tau)+1> 0 \). Using Lemma \ref{le2.3}, it follows that
\[
\begin{aligned}
J_2 &= \int_0^1 \frac{(1 - r)^{p+\a+1}}{(1 - tr)^{p+\a+2}(1 - |b|r)|1 - t\overline{b}r^2|^p} dr \\
&\leq \int_0^1 \frac{(1 - r)^{p+\a+1}}{(1 - t)^{p\delta}(1 - tr)^{p+\a+2-p\delta}(1 - |b|)^{p\tau}(1 - |b|r)^{ 1- p\tau}|1 - t\overline{b}r^2|^p} dr \\
&\leq \frac{1}{(1 - t)^{p\delta}(1 - |b|)^{p\tau}} \int_0^1 \frac{(1 - r)^{p+\a+1}}{(1 - r)^{p+\a+3- p\delta- p\tau}|1 - t\overline{b}r^2|^p} dr \\
&\lesssim \frac{1}{(1 - t)^{p\delta}(1 - |b|)^{p\tau}} \int_0^1 \frac{(1 - r)^{p(\delta+\tau)-2}}{(1 - t|b|r)^p} dr\\
&\lesssim \frac{1}{(1 - t)^{p\delta}(1 - |b|)^{p\tau}(1 - t|b|)^{p-p(\delta+\tau)+1}}.
\end{aligned}
\] 
Since \( \mu \) is a Carleson measure, using Lemma \ref{le2.4}, we get 
\[
\begin{aligned}
	&\sup_{b \in \mathbb{D}} (1-|b|^2) \left(\int_0^1 J_2^{\frac{1}{p}} d\mu(t) \right)^p\\
	\lesssim &
	\sup_{b \in \mathbb{D}} \left( \int_0^1 \frac{(1 - |b|)^{\frac{1}{p}-\tau}}{(1 - t)^{\delta} (1 - t|b|)^{1-(\delta+\tau)+\frac{1}{p}}} d\mu(t) \right)^p
	\lesssim 1.
\end{aligned}
\]
Therefore, 
\[
\| \mathcal{C}_\mu (f) \|_{AT_p^\infty(\a)} \lesssim \| f \|_{AT_p^\infty(\a)}.\]

{\bf Case \( p =1 \).}
Using Lemma \ref{le2.2} (ii), we have
\[
\begin{aligned}
&\int_0^1 \frac{(1 - r)^{p+\a+1}}{(1 - tr)^{p+\a+2}} \int_0^{2\pi} \frac{d\theta}{|1 - tre^{i\theta}|^p |1 - \overline{b}re^{i\theta}|^{2}} dr \\
\asymp& \int_0^1 \frac{(1 - r)^{p+\a+1}}{(1 - tr)^{p+\a+2}}
\left( \frac{1}{(1-|b|r) |1 - t\overline{b}r^2|} +\frac{1}{|1 - t\overline{b}r^2|^2}\log \frac{e}{1 - |\varphi_{br}(tr)|^2} \right) dr\\
=& \int_0^1 \frac{(1 - r)^{p+\a+1}}{(1 - tr)^{p+\a+2}(1-|b|r) |1 - t\overline{b}r^2|} dr + \int_0^1 \frac{(1 - r)^{p+\a+1}  \log \frac{e}{1 - |\varphi_{br}(tr)|^2} }{(1 - tr)^{p+\a+2} |1 - t\overline{b}r^2|^2}dr \\
=&J_3 + J_4.
\end{aligned}
\]
The estimation of \(J_3\) follows the same procedure as that of \(J_1\) and we obtain
\[
\begin{aligned}
 \sup_{b \in \mathbb{D}} (1-|b|^2) \left(\int_0^1 J_3^{\frac{1}{p}} d\mu(t) \right)^p 
\lesssim 1.
\end{aligned}
\]   

 Finally, we estimate \(J_4\).
Let \( 0 < d < \frac{1}{4} \). It is obvious that
\[
(1 - |\varphi_{br}(tr)|^2)^d \log \frac{e}{1 - |\varphi_{br}(tr)|^2} \lesssim 1.
\]
Since \( 2d + \frac{1}{2} < 1 \), we may choose a positive real number \( \gamma \) such that \( 2d + \frac{1}{2} < \gamma < 1 \). This yields that \(\gamma-1-2d > -1 \) and \( 2-2d - (\gamma-1-2d) - 1 =2-\gamma> 0 \). Bear in mind that \( 1 - t < 1 - tr \), \(1-r<1-tr\) and \( 1 - r < 1 - |b|r \). Using Lemma \ref{le2.3}, it follows that  
\[
\begin{aligned}
J_4 &= \int_0^1 \frac{(1 - r)^{p+\a+1} }{(1 - tr)^{p+\a+2} |1 - t\overline{b}r^2|^2}\log \frac{e}{1 - |\varphi_{br}(tr)|^2} dr \\
&\lesssim \int_0^1 \frac{(1 - r)^{p+\a+1}}{(1 - tr)^{p+\a+2} |1 - t\overline{b}r^2|^2 (1 - |\varphi_{br}(tr)|^2)^d} dr \\
&= \int_0^1 \frac{(1 - r)^{p+\a+1}}{(1 - tr)^{p+\a+2+d} (1 - |b|r)^d |1 - t\overline{b}r^2|^{2 - 2d}} dr\\
&\leq  \int_0^1 \frac{(1 - r)^{p+\a+1}}{(1-tr)^{p+\a+2+d-\gamma}(1 - t)^\gamma (1-r)^d |1 - t\ol{b}r^2|^{2 - 2d}} dr \\
&\le
\frac{1}{(1 - t)^\gamma}\int_0^1
\frac{(1 - r)^{\gamma-1-2d}}{(1 - t|b|r)^{2-2d}} dr\\
&\lesssim \frac{1}{(1 - t)^\gamma (1 - t|b|)^{2-\gamma}}.
\end{aligned}
\]
Since \( \mu \) is a Carleson measure, using Lemma \ref{le2.4}  we obtain
\[
\begin{aligned}
&\sup_{b \in \mathbb{D}} (1-|b|^2) \left(\int_0^1 J_4^{\frac{1}{p}} d\mu(t)\right)^p\\ 
\lesssim &
\sup_{b \in \mathbb{D}}\left( \int_0^1 \frac{(1 - |b|)^{\frac{1}{p}}}{(1 - t)^{\frac{\gamma}{p}} (1 - t|b|)^{1+\frac{1}{p}-\frac{\gamma}{p}}} d\mu(t)\right)^p  
\lesssim  1.
\end{aligned}
\]
Hence,
\[
\| \mathcal{C}_\mu (f) \|_{AT_p^\infty(\a)} \lesssim \| f \|_{AT_p^\infty(\a)},
\] 
which implies that \(\mathcal{C}_\mu\) is bounded on \(AT_p^\infty(\a)\). 
The proof is complete.
\end{proof}

\begin{Theorem}
Let \(1 \leq p<\infty\) and \(\a>-2\). Let \(\mu\) be a finite positive Borel measure on $[0, 1)$. The Ces\`aro-like operator \(\mathcal{C}_\mu\) is compact on \(AT_p^\infty(\a)\) if and only if \(\mu\) is a vanishing Carleson measure.
\end{Theorem}

\begin{proof} First we assume that \( \mathcal{C}_\mu \) is compact on \( AT_p^\infty(\a) \). For \( \frac{1}{2} < \rho < 1 \), let  
\[
f_\rho(z) = \frac{(1 - \rho)}{(1 - \rho z)^{ \frac{\a+2}{p}+1}}, \quad z \in \mathbb{D}.
\]  
We see that \( f_\rho \in AT_p^\infty(\a) \) and \( \sup_{\frac{1}{2} < \rho < 1} \| f_\rho \|_{AT_p^\infty(\a)} \lesssim 1 \). 
Furthermore,   \(f_\rho\to 0\) uniformly on compact subsets of \(\mathbb{D}\) as \(\rho\to 1\).
Since \( \mathcal{C}_\mu \) is compact on \( AT_p^\infty(\a) \), by virtue of  Lemma \ref{com}, it follows that  \[ \lim_{\rho \to 1} \| \mathcal{C}_\mu (f_\rho) \|_{AT_p^\infty(\a)} = 0  .\] 
Using Lemma \ref{az}, we have
\begin{align}\label{4.1}
\sup_{z \in \mathbb{D}} (1 - |z|^2)^{\frac{\a+2}{p}} |\mathcal{C}_\mu (f_\rho)(z)| \lesssim \| \mathcal{C}_\mu (f_\rho) \|_{AT_p^\infty(\a)} \to 0, \quad \text{as} \quad \rho \to 1. 
\end{align}
Thus, for \( \frac{1}{2}< \rho < 1 \),  by the proof of    Theorem \ref{cmub}  we have 
\[\begin{aligned} 
 (1-\rho)^{ \frac{\a+2}{p}}|\mathcal{C}_\mu (f_\rho)(\rho)|  \gtrsim  \frac{\mu([\rho, 1))}{1-\rho},
\end{aligned}\]
which implies that  
\[
\frac{\mu([\rho, 1))}{1 - \rho} \lesssim \| \mathcal{C}_\mu (f_\rho) \|_{AT_p^\infty(\a)}.
\]  
Combining with (\ref{4.1}), we get that \(\mu\) is a vanishing Carleson measure.

Conversely, suppose  \(\mu\) is a vanishing Carleson measure. Consider a bounded sequence \(\{f_j\}_{j=1}^\infty\) in \(AT_p^\infty(\a)\) that converges uniformly to 0 on every compact subset of \(\mathbb{D}\). We may assume without loss of generality that \(f_j(0) = 0\) for all \(j \geq 1\) and \(\sup_{j \geq 1} \|f_j\|_{AT_p^\infty(\a)} \lesssim 1\). Using Lemma \ref{com}, we need to establish that
$$
\lim_{j \to \infty} \|C_\mu(f_j)\|_{AT_p^\infty(\a)} = 0.
$$	
Since \( \mu \) is a vanishing Carleson measure, for any \( \epsilon > 0 \), using Lemma \ref{le4.2}, we obtain that there exists a \(  \delta \in (0,1)  \) such that  
\[
\int_0^1 \frac{(1 - |b|)^r}{(1 - t)^c (1 - t|b|)^{1 + r - c}} d\mu(t) < \epsilon  \quad \text{ for all } \quad \delta < |b| < 1,
\]  
where \(r>0 \) and \( 0 \leq c < 1 \). 
Observe that Lemma \ref{le4.2} also indicates that \( \int_0^1 \frac{1}{(1 - t)^c} d\mu(t) < \infty \). As a result, there exists a $t_0$ with \( 0 < t_0 < 1 \) for which  
\begin{align}\label{4.2}
\int_{t_0}^1 \frac{1}{(1 - t)^c} d\mu(t) < \epsilon.
\end{align}
Using Lemma \ref{at}, we get
\[
\begin{aligned}
 \| \mathcal{C}_\mu (f_j) \|_{AT_p^\infty(\a)}^p  
\lesssim& \sup_{|b| \leq \delta}  \int_{\mathbb{D}} |\mathcal{C}_\mu (f_j)'(z)|^p (1 - |z|^2)^{p+\a} (1 - |\varphi_b(z)|^2) dA(z)\\
&+ \sup_{\delta < |b| < 1}  \int_{\mathbb{D}} |\mathcal{C}_\mu (f_j)'(z)|^p (1 - |z|^2)^{p+\a} (1 - |\varphi_b(z)|^2) dA(z) \\
=& H_1+H_2.
\end{aligned}
\]
For \(H_2\), by Theorem \ref{cmub}, we obtain that
\[	
H_2 \lesssim \sup_{\delta < |b| < 1} \int_0^1 \frac{(1 - |b|)^r}{(1 - t)^c (1 - t|b|)^{1 + r - c}} d\mu(t) < \epsilon
\]
for some \(r>0 \) and \( 0 \leq c < 1 \).
Furthermore, using (\ref{xy}), we get
\[
\begin{aligned}
H_1 =&\sup_{|b| \leq \delta}  \int_{\mathbb{D}} |\mathcal{C}_\mu (f_j)'(z)|^p (1 - |z|^2)^{p+\a} (1 - |\varphi_b(z)|^2) dA(z)  \\
\leq& \sup_{|b| \leq \delta}  \int_{\mathbb{D}} \left( \int_0^1 \left( \frac{|f_j(tz)|}{|1 - tz|^2} + \frac{|f_j'(tz)|}{|1 - tz|} \right) d\mu(t) \right)^p (1 - |z|^2)^{p+\a} (1 - |\varphi_b(z)|^2) dA(z)  \\
\lesssim& \sup_{|b| \leq \delta}  \int_{\mathbb{D}} \left( \int_0^{t_0} \left( \frac{|f_j(tz)|}{|1 - tz|^2} + \frac{|f_j'(tz)|}{|1 - tz|} \right) d\mu(t) \right)^p (1 - |z|^2)^{p+\a} (1 - |\varphi_b(z)|^2) dA(z)  \\
&+ \sup_{|b| \leq \delta}  \int_{\mathbb{D}} \left( \int_{t_0}^1 \left( \frac{|f_j(tz)|}{|1 - tz|^2} + \frac{|f_j'(tz)|}{|1 - tz|} \right) d\mu(t) \right)^p (1 - |z|^2)^{p+\a} (1 - |\varphi_b(z)|^2) dA(z).
\end{aligned}
\] 
By virtue of the Cauchy integral theorem, it can be deduced that the sequence \(\{f_j'\}_{j = 1}^{\infty}\) converges uniformly to \(0\) on every compact subset of \(\mathbb{D}\). 
Hence,
\[
\begin{aligned}
&\sup_{|b| \leq \delta}  \int_{\mathbb{D}} \left( \int_0^{t_0} \left( \frac{|f_j(tz)|}{|1 - tz|^2} + \frac{|f_j'(tz)|}{|1 - tz|} \right) d\mu(t) \right)^p (1 - |z|^2)^{p+\a} (1 - |\varphi_b(z)|^2) dA(z)  \\
\lesssim& \sup_{|w| \leq t_0} (|f_j(w)| + |f_j'(w)|) \to 0, \quad  j \to \infty.
\end{aligned}
\]
Similar to the proof of Theorem \ref{cmub}, we can also show that 
\[
\begin{aligned}
&\sup_{|b| \leq \delta}  \int_{\mathbb{D}} \left( \int_{t_0}^1 \left( \frac{|f_j(tz)|}{|1 - tz|^2} + \frac{|f_j'(tz)|}{|1 - tz|} \right) d\mu(t) \right)^p (1 - |z|^2)^{p+\a} (1 - |\varphi_b(z)|^2) dA(z)  \\
\lesssim &\sup_{|b| \leq \delta} (1-|b|^2)  \int_{\mathbb{D}} \left( \int_{t_0}^1 \frac{d\mu(t)}{|1 - tz|(1 - t|z|)^{ \frac{\a+2}{p}+1} } \right)^p \frac{(1 - |z|^2)^{p+\a+1}}{|1 - \overline{b}z|^{2}} dA(z)  \\
\lesssim &\sup_{|b| \leq \delta} (1-|b|^2) \left( \int_{t_0}^1 \left( \int_{\mathbb{D}} \frac{(1 - |z|^2)^{p+\a+1} dA(z)}{|1 - tz|^p(1 - t|z|)^{p+\a+2}  |1 - \overline{b}z|^{2}} \right)^{\frac{1}{p}} d\mu(t) \right)^p\\
\lesssim& \sup_{|b| \leq \delta}\left( \int_{t_0}^1 \frac{(1 - |b|)^r}{(1 - t)^c (1 - t|b|)^{1 + r- c}} d\mu(t) \right)^p
\end{aligned}
\]
for some \(r>0 \) and \( 0 \leq c < 1 \). By (\ref{4.2}), we deduce that
\[
\sup_{|b| \leq \delta} \int_{t_0}^1 \frac{(1 - |b|)^r}{(1 - t)^c (1 - t|b|)^{1 + r - c}} d\mu(t) \lesssim \int_{t_0}^1 \frac{1}{(1 - t)^c} d\mu(t) < \epsilon.
\] 
Therefore,  
\[
\lim_{j \to \infty} \| \mathcal{C}_\mu (f_j) \|_{AT_p^\infty(\a)} = 0.
\]  
The proof is complete.
\end{proof}

\noindent \textbf{Acknowledgments.} 
 The   authors are supported  by GuangDong Basic and Applied Basic Research Foundation (No. 2023A1515010614).\msk

\noindent\textbf{Data Availability.}  Data sharing is not applicable for this article as no datasets were generated or analyzed during the current study.\msk

 \noindent \textbf{Conflict of interest.}   The authors declare no competing interests.\msk

\end{document}